\documentclass[a4paper]{article}

\usepackage[english]{babel}
\usepackage[utf8x]{inputenc}
\usepackage{amsmath}

\usepackage[a4paper,top=3cm,bottom=2cm,left=3cm,right=3cm,marginparwidth=1.75cm]{geometry}

\usepackage{amsthm}
\usepackage{amsfonts}
\usepackage{graphicx}
\usepackage[colorlinks=true, allcolors=blue]{hyperref}
\usepackage{url}

\newtheorem{definition}{Definition}
\newtheorem{lemma}{Lemma}
\newtheorem{proposition}{Proposition}
\newtheorem{theorem}{Theorem}
\newtheorem{remark}{Remark}
\newtheorem{example}{Example}

\def\NN{\mathbb{N}}
\def\QQ{\mathbb{Q}}

\def\ZZ{\mathbb{Z}}

\title{On p--adic Multidimensional Continued Fractions}

\author{Nadir Murru and Lea Terracini \\
Department of Mathematics G. Peano, University of Torino\\
Via Carlo Alberto 10, 10123, Torino, ITALY\\
nadir.murru@unito.it, lea.terracini@unito.it}
\date{}

\begin{document}
\maketitle

\begin{abstract}
Multidimensional continued fractions (MCFs) were introduced by Jacobi and Perron in order to generalize the classical continued fractions. In this paper, we propose an introductive fundamental study about MCFs in the field of the $p$--adic numbers  $\mathbb Q_p$. First, we introduce them from a formal point of view, i.e., without considering a specific algorithm that produces the partial quotients of a MCF, and we perform a general study about their convergence in $\mathbb Q_p$. In particular, we derive some conditions about their convergence and we prove that convergent MCFs always strongly converge in $\mathbb Q_p$ contrarily to the real case where strong convergence is not ever guaranteed. Then, we focus on a specific algorithm that, starting from a $m$--tuple of numbers in $\mathbb Q_p$, produces the partial quotients of the corresponding MCF. We see that this algorithm is derived from a generalized $p$--adic Euclidean algorithm and we prove that it always terminates in a finite number of steps when it processes rational numbers.
\end{abstract}


\section{Introduction}
Continued fractions are very classical objects in number theory. In the general case, a continued fraction is an expression of the form
\begin{equation} \label{eq:gen-cf} a_0 + \cfrac{b_1}{a_1 + \cfrac{b_2}{a_2 + \ddots}} \end{equation}
where the $a_n$'s and $b_n$'s are complex numbers or, more generally, belong to some topological field. When $b_n = 1$ and $a_n$'s are integer numbers, for any $n$, the previous expression is called simple continued fraction and is usually denoted by $[a_0, a_1, a_2, \dots]$. This kind of continued fractions has been extensively studied for their useful and beautiful properties. Indeed, they yield to a representation for any real number $\alpha$ where the $a_n$'s, called partial quotients, can be obtained by iterating the following algorithm:
$$\begin{cases} a_n = [\alpha_n] \cr \alpha_{n + 1} = \cfrac{1}{\alpha_k - a_k} \quad \text{if $\alpha_k$ is not integer}  \end{cases}, k = 0, 1, 2, ...,$$ 
starting with $\alpha_0 = \alpha$. The expansion provided by this algorithm is finite if and only if $\alpha$ is a rational number. Moreover, the expansion in  simple continued fraction of a real number is periodic if and only if the starting number is a quadratic irrationality. This beautiful property motivated Hermite \cite{Her} to ask to Jacobi for finding an algorithm that becomes eventually periodic when processes other kinds of algebraic irrationals. Jacobi \cite{Jac} proposed an algorithm for working with cubic irrationals and Perron \cite{Per} generalized it in any dimension. The Jacobi--Perron algorithm processes a $m$--tuple of real numbers $(\alpha_0^{(1)}, ..., \alpha_0^{(m)})$ and provides a $m$--tuple of integer sequences $(a_n^{(1)})_{n=0}^\infty, ..., (a_n^{(m)})_{n=0}^\infty$ by means of the following equations:

\begin{equation} \label{eq:alg-jp}
\begin{cases} a_n^{(i)} = [\alpha_n^{(i)}], \quad i = 1, ..., m, \cr
\alpha_{n+1}^{(1)} = \cfrac{1}{\alpha_n^{(m)} - a_n^{(m)}}, \cr
\alpha_{n+1}^{(i)} = \cfrac{\alpha_n^{(i-1)} - a_n^{(i-1)}}{\alpha_n^{(m)} - a_n^{(m)}}, \quad i = 2, ..., m, \end{cases} n = 0, 1, 2, ...
\end{equation}

In this way, a $m$--tuple of real numbers is represented by a $m$--tuple of integer sequences (the partial quotients) by means of objects that generalize the classical continued fractions and that we call multidimensional continued fractions. Indeed, the following relations hold:
$$\begin{cases} \alpha_n^{(i-1)} = a_n^{(i-1)} + \cfrac{\alpha_{n+1}^{(i)}}{\alpha_{n+1}^{(1)}}, \quad i = 2, ..., m \cr
\alpha_n^{(m)} = a_n^{(m)} + \cfrac{1}{\alpha_{n+1}^{(i)}} \end{cases}, n = 0, 1, 2, ...$$
The problem regarding the periodicity of multidimensional continued fractions is still open, i.e., it is not known if given an algebraic irrational of degree $m + 1$ there exist $m - 1$ real numbers such that the Jacobi--Perron algorithm is eventually periodic. Multidimensional continued fractions have been extensively studied during the years \cite{Ber}, \cite{Das}, \cite{Gar}, \cite{Ger}, \cite{Hen}, \cite{Kar}, \cite{Mur}, \cite{Raj}, \cite{Schw}, \cite{Tam}.

Classical continued fractions have been also studied in the field of the $p$--adic numbers $\mathbb Q_p$ aiming at obtaining results analogous to those holding in the real field (regarding in particular finiteness, periodicity and approximation). However, until now, there is not a standard definition of continued fractions in $\mathbb Q_p$, since a straightforward generalization does not exist. Schneider \cite{Sch} focused on general continued fractions of the form \eqref{eq:gen-cf}, where $b_n$'s are certain powers of the prime number $p$ and any $a_n \in \{1, 2, ..., p-1\}$. Ruban \cite{Rub} modified the Schneider's approach in order to work with simple continued fractions. Browkin \cite{Bro1} proposed an algorithm that produces simple continued fractions whose partial quotients belong to $\mathbb Z_p [\frac{1}{p}] \cap \left( -\frac{p}{2}, \frac{p}{2} \right)$. For an extensive survey on the subject see \cite{Mil}.

Schneider's continued fractions provide a finite or periodic expansion for rational numbers (see, e.g, \cite{Poo}). Moreover they are not always periodic for square roots of rational integers \cite{Til}, \cite{Weg}. 

Similarly, negative rational numbers have periodic expansion as Ruban's continued fractions and positive rational numbers have a finite expansion \cite{Lao}, \cite{Wan}. Recently, Capuano et al. \cite{Cap} have given a criterion for periodicity of the expansion of quadratic irrationalities in Ruban's continued fractions. Moreover, some studies about transcendence criteria for Ruban's continued fractions can be found in \cite{Oot}, \cite{Ubo}, \cite{Wan}.

Browkin's continued fractions provide a finite expansion for all rational numbers. However, also in this case, they do not provide a periodic expansion for all quadratic irrationality. In \cite{Bro2}, Browkin provides some new algorithms aiming at characterizing all square roots of rational integers with a periodic expansion, without obtaining a definitive result. Other studies regarding periodicity of Browkin's continued fractions can be found in \cite{Bed1}, \cite{Bed2}.

Further results about continued fractions in the $p$--adic field can be found in \cite{Han}, \cite{Koj}, \cite{Sai}.

Considering the many studies about multidimensional continued fractions over the real numbers, it seems interesting to study what happens in the $p$--adic field. At the best of our knowledge, the only works about $p$--adic multidimensional continued fractions are \cite{Sai2} and \cite{Tam2}. However, in these papers the authors only focus on the proposal of new algorithms with the aim of obtaining a multidimensional $p$--adic version of the Lagrange theorem. Thus, an in--depth investigation and a comprehensive presentation of this topic are still lacking.

In this paper, we would like to overcome this lack providing a fundamental study of $p$--adic multidimensional continued fractions. First of all, we present multidimensional continued fractions from a formal point of view in section \ref{sec:formal}, in order to perform a general study about their convergence in $\mathbb Q_p$. In section \ref{sec:conv}, we provide a sufficient condition for the convergence that does not depend on a specific algorithm. Moreover, we give some results useful for the construction of algorithms that generate convergent multidimensional continued fractions in $\mathbb Q_p$. In section \ref{sec:alg}, we focus on a specific algorithm, that we will call $p$--adic Jacobi--Perron algorithm and generalize the Browkin approach to the multidimensional case. In particular, this algorithm results to be equivalent to a generalized $p$--adic Euclidean algorithm. Moreover, we prove that if it terminates in a finite number of steps, then the processed input consists of elements in $\mathbb Q_p$ that are $\mathbb Q$--linearly dependent. On the other hand, we will prove that if the input consists of rational numbers, then the $p$--adic Jacobi--Perron algorithm terminates in a finite number of steps. Finally, we present some examples of the application of the $p$--adic Jacobi--Perron algorithm. This paper lays the basis for a more extensive study on $p$--adic multidimensional continued fractions. In particular, the following topics, which will be studied in a forthcoming paper, seem to be of relevant interest: a complete characterization of the inputs for which the algorithm stops; a study about periodicity related to algebraic irrationalities; an analysis of the quality of approximations of $p$--adic numbers.

\section{Multidimensional continued fractions from a formal point of view} \label{sec:formal}

In this section, we would like to give an original presentation of multidimensional continued fractions. Indeed, they are usually introduced starting from an algorithm, like, e.g., the Jacobi--Perron algorithm \eqref{eq:alg-jp}, working on real numbers. Here, we give all the possible definitions considering the partial quotients as formal quantities, which are not necessarily obtained by the application of an algorithm. Moreover, we also consider a more general case where the numerators of a multidimensional continued fraction are not necessarily equal to 1, but they are a general sequence, like the case of classical continued fractions in the form \eqref{eq:gen-cf}.

In the following we work in a field $\mathbb K$ and we fix an integer $m\geq 1$ (which will represent the dimension of the multidimensional continued fraction). Let us consider $m + 1$ sequences of indeterminates, whose elements are called \emph{partial quotients},
\begin{equation} \label{eq:pq}
(a^{(i)}_n)_{n\in\NN}=(a^{(i)}_0, a^{(i)}_1, a^{(i)}_2, \ldots , a^{(i)}_n, \ldots),\quad 
i=1, \ldots, m+1,
\end{equation}
with $a^{(m+1)}_0=1$, where these sequences are all infinite or all finite with the same length. We also consider $m + 1$ sequences of indeterminates, called \emph{complete quotients},

\begin{equation} \label{eq:cq}
(\alpha^{(i)}_n)_{n\in\NN}=(\alpha^{(i)}_0, \alpha^{(i)}_1, \alpha^{(i)}_2, \ldots , \alpha^{(i)}_n, \ldots),\quad 
i=1, \ldots, m+1.
\end{equation}
Partial quotients and complete quotients will be subjected to the relations
\begin{equation} \label{eq:mcf}
\alpha_n^{(i)} = a_n^{(i)} + \cfrac{\alpha_{n + 1}^{(i + 1)}}{\alpha_{n + 1}^{(1)}}, \quad i = 1, \ldots, m, 
\end{equation}
\begin{equation} \label{eq:mcf2}
\alpha_n^{(m + 1)} = a_n^{(m + 1)}
\end{equation}
for any $n \geq 0$. This defines the \emph{multidimensional continued fraction} (MCF) of dimension $m$.

\begin{example}
Fixing $m = 2$ and three sequences $(a_0^{(1)}, a_1^{(1)}, ...)$, $(a_0^{(2)}, a_1^{(2)}, ...)$, $(a_0^{(3)}, a_1^{(3)}, ...)$, equations \eqref{eq:mcf} and \eqref{eq:mcf2} determine the MCF of dimension 2 that has the following expansion:

\begin{equation*} 
\alpha_0^{(1)}=a_0^{(1)}+\cfrac{a_1^{(2)}+\cfrac{a_2^{(3)}}{a_2^{(1)}+\cfrac{a_3^{(2)}+\cfrac{a_4^{(3)}}{\ddots}}{a_3^{(1)}+\cfrac{\ddots}{\ddots}}}}{a_1^{(1)}+\cfrac{a_2^{(2)}+\cfrac{a_3^{(3)}}{a_3^{(1)}+\cfrac{\ddots}{\ddots}}}{a_2^{(1)}+\cfrac{a_3^{(2)}+\cfrac{a_4^{(3)}}{\ddots}}{a_3^{(1)}+\cfrac{\ddots}{\ddots}}}} \quad \text{and} \quad \alpha_0^{(2)}=a^{(2)}_0+\cfrac{a_1^{(3)}}{a_1^{(1)}+\cfrac{a_2^{(2)}+\cfrac{a_3^{(3)}}{a_3^{(1)}+\cfrac{\ddots}{\ddots}}}{a_2^{(1)}+\cfrac{a_3^{(2)}+\cfrac{a_4^{(3)}}{\ddots}}{a_3^{(1)}+\cfrac{\ddots}{\ddots}}}}.
\end{equation*}
\end{example}

\begin{remark}
Following this approach, a MCF can be viewed as an object in the polynomial ring \\ $\mathbb K[a_n^{(i)}, \alpha_n^{(i)} / i = 1, \ldots, m+1, n \in \mathbb N] / I$, $n \in \NN$, where $I$ is the ideal generated by the polynomials $\alpha_n^{(m+1)} - a_n^{(m+1)}$ and $\alpha_{n+1}^{(1)}(\alpha_n^{(i)} - a_n^{(i)}) - \alpha_{n+1}^{(i+1)}$.
\end{remark}

\begin{remark}
If the sequences of partial quotients \eqref{eq:pq} are finite, they must have the same length $r$ and in this case also the sequences of complete quotients \eqref{eq:cq} are finite and $\alpha_r^{(i)} = a_r^{(i)}$, for any $i = 1, \ldots, m$.
\end{remark}

From equations \eqref{eq:mcf} and \eqref{eq:mcf2}, the following equalities hold true:
\begin{equation} \label{eq:alg-mcf}
\begin{cases}
\alpha_{n + 1}^{(1)} = \cfrac{a_{n + 1}^{(m+1)}}{\alpha_n^{(m)} - a_n^{(m)}} \cr \cr
\alpha_{n + 1}^{(i)} = a_{n + 1}^{(m+1)} \cdot \cfrac{(\alpha_n^{(i - 1)} - a_n^{(i - 1)})}{\alpha_n^{(m)} - a_n^{(m)}} = \alpha_{n + 1}^{(1)}(\alpha_n^{(i - 1)} - a_n^{(i - 1)}), \quad i = 2, \ldots, m
\end{cases}
\end{equation}
for any $n \geq 0$.

Now, we can define the $n$--\emph{th convergents} of a multidimensional continued fraction as
$$
Q^{(i)}_n=\frac {A^{(i)}_{n}}{A^{(m+1)}_{n}},
$$
for $i=1,\ldots, m$ and $n\in\NN$, where
\begin{equation*} 
A^{(i)}_{-j} = \delta_{ij}, \quad A^{(i)}_{0} = a^{(i)}_0, \quad A^{(i)}_{n} = \sum_{j=1}^{m+1}a^{(j)}_{n}A^{(i)}_{n-j}
\end{equation*}
for $i = 1, \ldots, m + 1$, $j = 1, \ldots, m$ and any $n \geq 1$, where $\delta_{ij}$ is the Kronecker delta. It can be proved by induction that for every $n \geq 1$ and $i = 1, \ldots, m$, we have
\begin{equation} \label{eq:alpha0}
\alpha_0^{(i)}=\frac {\alpha_n^{(1)}A^{(i)}_{n-1}+ \alpha_n^{(2)}A^{(i)}_{n-2}+\ldots +\alpha_n^{(m+1)}A^{(i)}_{n-m-1} }{\alpha_n^{(1)}A^{(m+1)}_{n-1}+ \alpha_n^{(2)}A^{(m+1)}_{n-2}+\ldots +\alpha_n^{(m+1)}A^{(m+1)}_{n-m-1} }
\end{equation}
Let us note that numerators and denominators of the convergents can be computed by the products of the following matrices
\begin{eqnarray*}
\mathcal{A}_n &=& \begin{pmatrix} a_n^{(1)} &1 &0&\ldots &0\\
a_n^{(2)} &0 &1&\ldots &0\\
\vdots& \vdots& \vdots& \vdots& \vdots\\
a_n^{(m)} &0 &0&\ldots &1\\  a_n^{(m+1)} &0 &0&\ldots &0\end{pmatrix}, \quad \hbox{ for } n\geq 0.
\end{eqnarray*}
Indeed, if we define
$$
\mathcal{B}_n = \mathcal{B}_{n - 1} \mathcal{A}_n = \mathcal{A}_0 \mathcal{A}_1 \cdots \mathcal{A}_n, \quad \hbox{ for } n\geq 1,
$$
we have that
\begin{eqnarray*}
\mathcal{B}_n &= &\begin{pmatrix} {A^{(1)}_{n}} &{A^{(1)}_{n-1}} &\ldots & {A^{(1)}_{n-m-1}}\\
{A^{(2)}_{n}} &{A^{(2)}_{n-1}} &\ldots & {A^{(2)}_{n-m-1}}\\
\vdots &\vdots&\vdots& \vdots\\
{A^{(m+1)}_{n}} &{A^{(m+1)}_{n-1}} &\ldots & {A^{(m+1)}_{n-m-1}}\end{pmatrix} \quad\hbox{ for } n\geq 1.
\end{eqnarray*}
Notice that
\begin{equation*}
\det(\mathcal{A}_n) = (-1)^{m+1}a_n^{(m+1)}
\end{equation*}  
\begin{equation*}
\det(\mathcal{B}_n) = \prod_{j=0}^n (-1)^{m+1}a_j^{(m+1)} =  (-1)^{(n+1)(m+1)}\prod_{j=0}^na_j^{(m+1)}.
\end{equation*}   

\section{Convergence of multidimensional $p$-adic continued fractions} \label{sec:conv}

In this section we will focus on $\mathbb K = \mathbb Q_p$ the field of the $p$--adic numbers and we study here the convergence of a multidimensional continued fraction. Specifically, we give a sufficient condition concerning the partial quotients in order that the MCF converges in $\mathbb Q_p$. Moreover, we will see that under the same condition, we also have strong convergence against the real case where strong convergence of a MCF is not always guaranteed.

In the following, $\lvert \cdot \rvert$ will always denote the $p$--adic norm. Now, let us consider the following conditions on the partial quotients of a MCF:
\begin{equation}\label{eq:cond-conv}
\begin{cases}
\lvert a_n^{(1)} \rvert \geq 1 \cr
\lvert a_n^{(i)} \rvert < \lvert a_n^{(1)} \rvert, \quad i = 2, \ldots, m+1
\end{cases}
\end{equation}
for any $n \geq 1$.

\begin{lemma}\label{lem:norme} Under the conditions \eqref{eq:cond-conv}, we have
\begin{equation*}   |A_n^{(m+1)}|= \prod_{h=1}^n |a^{(1)}_h|.\end{equation*}
for any $n\geq 1$.
\end{lemma}
\begin{proof}  
We prove the lemma by induction on $n$. For $n=1$, $A^{(m+1)}_1=a_1^{(1)}$ and the result holds true. Now, let us consider $A^{(m+1)}_{n+1}$, we have
\begin{equation} \label{eq:inductivestep} 
|A^{(m+1)}_{n+1}|=|a_{n+1}^{(1)} A_n^{(m+1)} +\sum_{j=2}^{m+1} a_{n+1}^{(j)} A^{(m+1)}_{n+1-j}|.
\end{equation}
By inductive hypothesis $|a_{n+1}^{(1)} A_n^{(m+1)}|= \prod_{h=1}^{n+1} |a_h^{(1)}|$ and by \eqref{eq:cond-conv}, $|a_{n+1}^{(j)}|< |a_{n+1}^{(1)}|$ for $j\geq 2$. Therefore, for $j\geq 2$, we have
$$|a_{n+1}^{(j)}A_{n+1-j}^{(m+1)}|< |a_{n+1}^{(1)}A_{n+1-j}^{(m+1)}|=|a_{n+1}^{(1)} \prod_{h=1}^{n+1-j}a_h^{(1)}|,$$
and the latter is $\leq \prod_{h=1}^{n+1}| a_h^{(1)}|$ by \eqref{eq:cond-conv}. It follows that the right side norm   in (\ref{eq:inductivestep}) is equal to the norm of its first summand and the result follows. 
\end{proof}

\begin{lemma}\label{prop:generalfact}
Let $(z_n)_{n\in\NN}$ be a sequence of natural numbers such that there exists $m\geq 1$ for which
$$z_{n+m+1}>\min\{z_{n+1},\ldots, z_{n+m}\}\hbox{ for } n\gg 0,$$
then $z_n\to \infty$ for $n\to\infty$.
\end{lemma}

\begin{proof}
We put $M_k=\min\{z_{k-1},\ldots, z_{k-m}\}$, for $k\gg 0$, then $z_k > M_k$. Notice that $M_{k+1}\geq M_k$, in fact
\begin{itemize}
\item if $\min\{z_{k-1},\ldots, z_{k-m}\}=\min\{z_{k-1},\ldots, z_{k-m+1}\}$, then $M_k=M_{k+1}$;
\item otherwise  $\min\{z_{k-1},\ldots, z_{k-m}\}<\min\{z_{k-1},\ldots, z_{k-m+1}\}$ so that $\min\{z_{k},\ldots, z_{k-m+1}\}>z_{k-m}$ and $M_{k+1}>M_k$.
\end{itemize}
Assume that $\min\{z_{k-1},\ldots, z_{k-m}\}=z_{k-i}$ and $z_{k-j}>z_{k-i}$ for $j=1,\ldots, i-1$. Then
\begin{eqnarray*} z_{k-i}=\min\{z_{k-1},\ldots, z_{k-m}\}\\ =\min\{z_{k},\ldots, z_{k-m+1}\}\\ \vdots \\
=\min\{z_{k-i+m-1},\ldots, z_{k-i}\}\\
<\min\{z_{k-i+m},\ldots, z_{k-i+1}\}.
\end{eqnarray*}
It follows that $M_{k}<M_{k-i+m+1}$, so that for $n\to \infty$, $M_n\to\infty$ and $z_n\to\infty$.
\end{proof}

\begin{theorem}\label{teo:convergence} Under the conditions \eqref{eq:cond-conv}, the sequences $(Q_n^{(i)})_{n\in\NN}$ converge for $i=1,\ldots, m$.
\end{theorem}
\begin{proof} For $i=1,\ldots , m$ and $k$ sufficiently large, we have
\begin{eqnarray*}
|Q_{k+1}^{(i)}-Q_k^{(i)} | &=& \left |\frac {A^{(i)}_{k+1}A^{(m+1)}_{k}-A^{(i)}_{k}A^{(m+1)}_{k+1} }{A^{(m+1)}_{k+1}A^{(m+1)}_{k}}\right |\\
&=& \frac 1 {|A^{(m+1)}_{k+1}A^{(m+1)}_{k}|} \left |\sum_{j=1}^{m+1} a_{k+1}^{(j)} A_{k+1-j}^{(i)}A_k^{(m+1)} - \sum_{j=1}^{m+1} a_{k+1}^{(j)} A_{k+1-j}^{(m+1)}A_k^{(i)}\right |\\
&=& \frac 1 {|A^{(m+1)}_{k+1}|} \left | \sum_{j=2}^{m+1} a_{k+1}^{(j)} A_{k+1-j}^{(m+1)}\left (Q^{(i)}_{k+1-j}- Q^{(i)}_k  \right )\right |
\end{eqnarray*}
By Lemma \ref{lem:norme}, 
$$\left | a_{k+1}^{(j)} \frac{ A_{k+1-j}^{(m+1)}}{A_{k+1}^{(m+1)}}\right |= \frac {| a_{k+1}^{(j)}|}{\prod_{h=k+2-j}^{k+1} |a^{(1)}_h|}  $$
and the latter is $<1$ for $j\geq 2$ by \eqref{eq:cond-conv}.
Therefore we obtain
$$|Q_{k+1}^{(i)}-Q_k^{(i)} | < \max_{2\leq j\leq m+1}\left\{ \left|Q_k^{(i)}-Q_{k+1-j}^{(i)} \right |\right\}.$$
By the ultrametric inequality 
$$\left|Q_k^{(i)}-Q_{k+1-j}^{(i)} \right |\leq \max_{k+1-j<h\leq k}\left\{|Q_h^{(i)}-Q_{h-1}^{(i)}|\right\}$$
so that 
$$|Q_{k+1}^{(i)}-Q_k^{(i)} | < \max_{k-m< h\leq k}\left\{|Q_h^{(i)}-Q_{h-1}^{(i)}|\right\}.$$
Then, Lemma \ref{prop:generalfact} ensures that $|Q_{k+1}^{(i)}-Q_k^{(i)} |\to 0$ for $k\to \infty$, so that the sequence $(Q_n^{(i)})$ is a Cauchy sequence.
\end{proof}

Thus, we have proved that a MCF whose partial quotients satisfy \eqref{eq:cond-conv} converges $p$--adically. Until now, we started from partial quotients and studied the possible convergence of the associated MCF. Usually, MCFs are presented considering an input sequence of elements $(\alpha_0^{(1)}, \ldots, \alpha_0^{(m)})$ and the next complete quotients are iteratively evaluated by means of equations \eqref{eq:alg-mcf}. The partial quotients are obtained applying certain functions, which we will denote by $s^{(i)}$, to the complete quotients (e.g., in the real simple case, $s^{(i)}$'s are all equal to the floor function). In the next theorem, we see a condition about the functions $s^{(i)}$ in order to obtain a MCF that converges to a $m$--tuple $(\alpha_0^{(1)}, \ldots, \alpha_0^{(m)})$. In particular, the theorem provides some conditions for constructing several functions $s^{(i)}$ useful for constructing iterative algorithms. All the functions appearing in literature (as the function used by Browkin \cite{Bro2}, Ruban \cite{Rub}, Schneider \cite{Sch}, Tamura and Yasutomi \cite{Tam}) can be seen as particular realizations of the functions $s^{(i)}$ described below.

\begin{theorem} \label{thm:conv-alg-s} 
Let $(a_n^{(m+1)})_{n \in \NN}$ be a sequence of non--zero elements in $\QQ$, with $a_0^{(m+1)} = 1$. Assume to have a collection of functions  
$$s^{(i)}:\QQ_p\rightarrow \QQ,\  i=1,\ldots, m-1,\quad s_n^{(m)}:\QQ_p\rightarrow \QQ, \ n\in\NN$$ 
satisfying, for every $\gamma\in\QQ_p^*$ the following conditions
\begin{eqnarray}\label{eq:hypoalgo1}
|\gamma-s^{(i)}(\gamma)| &<&\min\{|\gamma|,1\}\hbox{ for }i=1,\ldots, m-1\\\label{eq:hypoalgo2}
|\gamma-s_n^{(m)}(\gamma)| &<&\min\{|\gamma|,1\}\hbox{ for } n\in\NN\\
\label{eq:hypoalgo3}
|\gamma-s_n^{(m)}(\gamma)| &\leq & |a_{n+1}^{(m+1)}|.
\end{eqnarray}
and $s^{(i)}(0) = s_n^{m}(0) = 0$ for  $i=1,\ldots, m-1, n\in\NN$.
For any $\alpha_0^{(i)} \in \mathbb Q_p$, for $i = 1, ..., m$, define, when possible,
\begin{equation} \label{eq:ai} \begin{cases}  a_n^{(i)} = s^{(i)}(\alpha_n^{(i)}) \hbox{ for } i=1,\ldots, m-1\cr \cr
a_n^{(m)} = s_n^{(m)}(\alpha_n^{(m)})\cr \cr
\alpha_{n+1}^{(1)} = \cfrac{a_{n+1}^{(m+1)}}{\alpha_n^{(m)} - a_n^{(m)}} \cr \cr
\alpha_{n+1}^{(i)} = a_{n+1}^{(m+1)} \cdot \cfrac{\alpha_n^{(i-1)} - a_n^{(i-1)}}{\alpha_n^{(m)} - a_n^{(m)}}, \quad i = 2, ..., m
\end{cases}
\end{equation}
for $n = 0, 1, 2, ...$. Then the MCF defined by the partial quotients $a_n^{(i)}$'s converges to $(\alpha^{(1)},\ldots, \alpha^{(m)})$.
\end{theorem}
\begin{proof}
First we show that the MCF defined by the partial quotients $a_n^{(i)}$'s converges, i.e., conditions \eqref{eq:cond-conv} are satisfied. By (\ref{eq:ai}) and (\ref{eq:hypoalgo3}), for $n\geq 1$, we have
$$|a_n^{(1)}|=|\alpha_n^{(1)}|=\left| \cfrac{a_{n}^{(m+1)}}{\alpha_{n-1}^{(m)} - a_{n-1}^{(m)}}\right |\geq 1. $$
Moreover, notice that the above equality gives also $|a_n^{(m+1)}|< |a_n^{(1)}|$ for $n\geq 1$. Again by (\ref{eq:ai}) and (\ref{eq:hypoalgo3}), for $i=2,\ldots, m$ and $n\geq 1$, we obtain
$$|a_{n}^{(i)}|=|\alpha_{n}^{(i)}|=|\alpha_n^{(1)}||{\alpha_{n-1}^{(i-1)} - a_{n-1}^{(i-1)}}|<|\alpha_n^{(1)}|=|a_{n}^{(1)}|.$$
Now, we prove that the convergents $Q_n^{(i)}$ have as limit the quantities $\alpha_0^{(i)}$, for $i=1, \ldots, m$. For $n$ sufficiently large, we can write
\begin{equation*}
\alpha_0^{(i)}=\frac{\sum_{j=1}^{m+1}\alpha_{n}^{(j)}A^{(i)}_{n-j}}{\sum_{j=1}^{m+1}\alpha_{n}^{(j)}A^{(m+1)}_{n-j}}
= \sum_{j=1}^{m+1} f^{(j)}_nQ_{n-j}^{(i)}
\end{equation*}
where
$$f_n^{(j)}=\frac{ \alpha_n^{(j)}A_{n-j}^{(m+1)}}{\sum_{j=1}^{m+1}\alpha_{n}^{(j)}A^{(m+1)}_{n-j}}.$$
Let us observe that hypotheses \eqref{eq:hypoalgo1} and \eqref{eq:hypoalgo2} ensure that $|\alpha_n^{(i)}| = |a_n^{(i)}|$, for $n \in \NN$, $i =1, \ldots, m$ and consequently $|\alpha_n^{(i)}| < |a_n^{(1)}|$. Thus, by Lemma \ref{lem:norme} and conditions \eqref{eq:cond-conv} we have
\begin{eqnarray*} |\alpha_n^{(1)}A_{n-1}^{(m+1)}|&=&\prod_{h=1}^{n}|a^{(1)}_h|,\\
|\alpha_n^{(j)}A_{n-j}^{(m+1)}| &= &|\alpha_n^{(j)}|\prod_{h=1}^{n-j}|a^{(1)}_h|< \prod_{h=1}^{n}|a^{(1)}_h|\hbox{ for } j=2,\ldots, m+1
\end{eqnarray*}
so that
$$|\sum_{j=1}^{m+1}\alpha_{n}^{(j)}A^{(m+1)}_{n-j}|= \prod_{h=1}^{n}|a^{(1)}_h|.$$
Therefore 
$$|f_n^{(1)}| =1\hbox{ and } |f_n^{(j)}|<1\hbox{ for } j=2,\ldots, m+1;$$
moreover by definition
$$\sum_{j=1}^{(m+1)} f^{(j)}_n=1.$$
For $i=1,\ldots, m$ let $\gamma^{(i)}=\lim_{n\to\infty} Q^{(i)}_n$. Then
\begin{eqnarray*}
\alpha_0^{(i)}-\gamma^{(i)} &=& \sum_{j=1}^{m+1} f^{(j)}_nQ_{n-j}^{(i)}-\sum_{j=1}^{m+1} f^{(j)}_n\gamma^{(i)}\\
&=& \sum_{j=1}^{m+1} f^{(j)}_n(Q_{n-j}^{(i)}-\gamma^{(i)})
\end{eqnarray*}
that converges to 0, for $n\to \infty$, because the $f_n^{(j)}$ are bounded.
\end{proof}

In the previous theorems, we have seen the weak convergence of a MCF. Furthermore, we can show that in the case of convergent $p$--adic MCFs we also have strong convergence (for the definition of strong convergence in the context of continued fractions see, e.g., \cite{Nak}). Let us define 
$$V^{(i)}_n=A^{(i)}_n-\alpha^{(i)}A^{(m+1)}_n,$$
for any $n\geq -m$ and $i=1, \ldots, m$. It is straightforward to see that 
\begin{equation}\label{eq:ricorsioneV}
V^{(i)}_n=\sum_{j=1}^{m+1}a_n^{(j)}V^{(i)}_{n-j}, \quad \sum_{j=1}^{m+1} \alpha^{(j)}_{n}V^{(i)}_{n-j}=0.
\end{equation}

\begin{theorem}\label{teo:strong convergence} 
Under the hypotheses of Theorem \ref{thm:conv-alg-s}, we have 
$$\lim_{n\to\infty }V^{(i)}_n=0,$$
for $i=1,\ldots m$.
\end{theorem}
\begin{proof} By \eqref{eq:ricorsioneV}, we have
$$V^{(i)}_n=\sum_{j=1}^{m+1}a_n^{(j)}V^{(i)}_{n-j}=\sum_{j=1}^{m+1}a_n^{(j)}V^{(i)}_{n-j}-\sum_{j=1}^{m+1}\alpha_n^{(j)}V^{(i)}_{n-j}=\sum_{j=1}^{m+1}(a_n^{(j)}-\alpha_n^{(j)}) V^{(i)}_{n-j}, $$
 for any $n\geq 1$. Moreover by the hypotheses on the partial and complete quotients, we have $|a_n^{(j)}-\alpha_n^{(j)}|<1$, so that
$$|V^{(i)}_n|\leq \max_{1\leq j\leq m+1}\{|a_n^{(j)}-\alpha_n^{(j)}|| V^{(i)}_{n-j}|\}< \max_{1\leq j\leq m+1}\{|V^{(i)}_{n-j}|\}.$$
Passing to valuations and applying Lemma \ref{prop:generalfact}, we obtain the result.
\end{proof}

\begin{remark}
It is interesting to observe that, in the case of $\mathbb K = \mathbb R$, i.e., when working with classical MCFs, strong convergence is not ever guaranteed if $m > 1$, see \cite{Dub}. On the contrary, in the case of $\mathbb K = \mathbb Q_p$, when convergence is guaranteed, we also have strong convergence.
\end{remark}

\section{A generalized Euclidean algorithm for multidimensional continued fractions in $\mathbb Q_p$} \label{sec:alg}

\subsection{An analogous of the Jacobi--Perron algorithm in $\mathbb Q_p$}
In the previous sections, we have introduced MCF from a formal point of view and studied the properties of convergence in the field of the $p$--adic numbers. Here, we focus on a specific algorithm that, starting from a $m$--tuple of numbers in $\mathbb Q_p$, produces the partial quotients of the corresponding MCF. As we will see, the MCF constructed by the algorithm has the partial quotients $a_n^{(m+1)}$ equal to 1 for any $n \geq 0$. Hence, first of all we see that given a MCF with general partial quotients, we can always get a MCF with $a_n^{(m+1)} = 1$, for $n \geq 0$.

\begin{proposition}\label{prop:deomogeneizza}
Let $(w_n)_{n\in \NN}$ be any sequence of non--zero elements in $\mathbb K$, and put $w_{k}=1$ for $k<0$. Let us consider the partial quotients $(a^{(i)}_n)_{n\in\NN}$ of a MCF, for $i=1,\ldots, m+1$, whose convergents are denoted by $Q_n^{(i)}$ and $A_n^{(i)}$ are, as usual, their numerators and denominators. Define
\begin{equation} \label{eq:pqt}
\tilde a^{(i)}_n=\frac {w_{n-i}}{w_n} a_n^{(i)}, \quad i= 1, \ldots, m+1, \quad n\in\NN.
\end{equation}
Let $\tilde{Q}_n^{(i)}$ be the convergents of the MCF defined by the partial quotients \eqref{eq:pqt} and $\tilde{A}_n^{(i)}$ be the numerators and denominators of these convergents. Then 
$$A_n^{(i)}=w_n \tilde{A}_n^{(i)},\quad \hbox{ for } n\in\NN \hbox{ and } i=1,\ldots m+1,$$
so that
$$Q_n^{(i)}=\tilde{Q}_n^{(i)},\quad \hbox{ for } n\in\NN \hbox{ and } i=1,\ldots m.$$
\end{proposition}

\begin{proof}
We prove the proposition by induction. We clearly have $A_n^{(i)}=w_n \tilde{A}_n^{(i)}$, for $n<0$, and $A_0^{(i)}=a_0^{(i)}=w_0\tilde{a}_0^{(i)}=w_0\tilde{A}_0^{(i)}$. For $n\geq 1$, we have
\begin{equation*}
A_n^{(i)}=\sum_{j=1}^{m+1}a_n^{(j)}A_{n-j}^{(i)}=\sum_{j=1}^{m+1}\left(\frac {w_n}{w_{n-j}} \tilde{a}_n^{(j)}\right ) A_{n-j}^{(i)}
\end{equation*}
\begin{equation*}
= \sum_{j=1}^{m+1}\left(\frac {w_n}{w_{n-j}} \tilde{a}_n^{(j)}\right ) \left ( w_{n-j}\tilde{A}_{n-j}^{(i)}\right )= w_n\sum_{j=1}^{m+1} \tilde{a}_n^{(j)}\tilde{A}_{n-j}^{(i)}= w_n\tilde{A}_n^{(i)}.
\end{equation*}
\end{proof}
Let us consider
$$w_n=\prod_{\substack{0<j\leq n\\
                              j\equiv n\bmod{m+1}
                               }}a_j^{(m+1)}, \quad n \in \NN,$$
where $(a^{(i)}_n)_{n\in\NN}$ partial quotients of a MCF as in \eqref{eq:pq}.
The sequence $(w_n)_{n\in\NN}$ can be also described inductively by
$$ w_n=\left\{
\begin{array}{ll}
1 & \hbox{ for  } n\leq 0\\
a_n^{(m+1)} w_{n-(m+1)} &\hbox{ for } n\geq 1
\end{array}
\right .
$$
Defining  $\tilde{a}_n^{(i)},\tilde{A}_n^{(i)},\tilde{Q}_n^{(i)}$ as in Proposition \ref{prop:deomogeneizza}, we get a MCF with partial quotients  $(a^{(i)}_n)_{n\in\NN}$, $i=1, \ldots, m+1$, having the same convergents as the MCF with partial quotients $(a^{(i)}_n)_{n\in\NN}$, and satisfying, for any $n\geq 0$, the following condition:
\begin{equation} \label{eq:deomo}
\tilde{a}^{(m+1)}_n=\frac {w_{n-(m+1)}}{w_n}{a}^{(m+1)}_n=1.
\end{equation}

\begin{remark}
Notice that, in general, transforming the partial quotients using \eqref{eq:deomo} does not preserve the convergence conditions \eqref{eq:cond-conv}. This fact shows that these conditions are sufficient, but not necessary for convergence.
\end{remark}

In the following we will always refer to MCFs whose partial quotients satisfy $a^{(m+1)}_n=1$, for any $n \geq 0$. Note, that in this case, conditions \eqref{eq:cond-conv} can be translated into
\begin{equation}\label{eq:cond-conv-v}
\begin{cases}
\lvert a_n^{(1)} \rvert > 1 \cr
\lvert a_n^{(i)} \rvert < \lvert a_n^{(1)} \rvert, \quad i = 2, \ldots, m
\end{cases}
\end{equation}
for any $n \geq 1$. These convergence conditions are the $p$--adic analogous of the Perron conditions for convergence of MCFs in the real field, see, e.g., \cite{DubA}. Moreover, from now on, $p$ will denote an odd prime. Now, we introduce a function over the $p$--adic numbers that will play the same role of the floor function over the reals.

\begin{definition} \label{def:browkins}
Given any $\alpha \in \mathbb Q_p$, uniquely written as
\begin{equation*}
\alpha=\sum_{j=k}^\infty x_jp^j, \quad k, x_j\in\ZZ, \quad x_j\in \left (-\frac p 2,\frac p 2\right),
\end{equation*}
the \emph{Browkin function} $s: \mathbb Q_p \longrightarrow \QQ$ is defined by
$$s(\alpha)= \sum_{j=k}^0 x_jp^j.$$
We call it \emph{Browkin function}, since $s$ has been introduced and used by Browkin \cite{Bro1} for studying $p$--adic continued fractions.
\end{definition}

The following proposition lists some property of $s$.

\begin{proposition}\label{prop:sproperties}
Given the Browkin function $s$, we have:
\begin{itemize}
\item[a)] $s(0)=0$, $s(-\alpha)=-s(\alpha)$ for every $\alpha\in\QQ_p$;
\item[b)] $s(\alpha)=s(\beta)$ if and only if $\alpha-\beta\in p\ZZ_p$, in particular $s$ is a locally constant map;
\item[c)] whatever topology is considered on $\QQ$, $s$ is a continuous map;
\item[d)] for every $\alpha\in\QQ_p$ either $s(\alpha)=0$ or $|s(\alpha)|\geq 1$;
\item[e)] if $s(\alpha)\not=0$ then $|s(\alpha)|=|\alpha|$;
\item[f)] $s(\QQ_p)=\ZZ\left [\frac 1 p\right ]\cap \left (-\frac p2,\frac p 2\right)$.
\end{itemize}
\end{proposition}
\begin{proof}
Properties $a)$, $b)$, $d)$ and $e)$ are straightforward. As for $c)$, notice that point $b)$ implies that $s$ is continuous when $\QQ$ is given the discrete topology. Since the latter is the finest possible topology on $\QQ$, we get the result. \\
We prove $f)$. 
If $\alpha=\sum_{j=-h}^0 x_jp^j$, then
$$|\alpha |_\infty \leq \frac {p-1} 2 \sum_{j=0}^h \frac 1 {p^j}=  \frac {p-1}  2\frac {1-\frac 1{p^{h+1}}} {1-\frac 1 p}<\frac p 2
$$
where $|\cdot|_\infty$ is the Euclidean norm. Conversely let $\alpha\in \ZZ\left [\frac 1 p\right ]\cap \left (-\frac p2,\frac p 2\right)$. Then $$\beta=\alpha-s(\alpha)\in \ZZ\left [\frac 1 p\right ]\cap p\ZZ_p=p\ZZ.$$ 
Suppose $\beta\not=0$;  then $|\beta |_\infty \geq p$ so that $|\alpha |=|s(\alpha)+\beta |> \frac p 2$, a contradiction. Therefore $\beta=0$, so that $\alpha=s(\alpha)\in s(\QQ_p)$.
\end{proof}

The Browkin function $s$ will play the same role as the floor function in \eqref{eq:alg-jp} for defining an algorithm that produces the partial quotients of a MCF convergent to a $m$--tuple $(\alpha^{(1)},\ldots, \alpha^{(m)})\in\QQ_p^m$. \\ \\
\noindent{\bf The p--adic Jacobi--Perron algorithm}
\begin{itemize}
\item Put $\alpha_0^{(i)}=\alpha^{(i)}$ for $i=1,\ldots , m$;
\item Define recursively
\begin{equation} \label{eq:alg} \begin{cases}  a_n^{(i)} = s(\alpha_n^{(i)}) \cr 
\alpha_{n+1}^{(1)} = \cfrac{1}{\alpha_n^{(m)} - a_n^{(m)}} \cr 
\alpha_{n+1}^{(i)} = \alpha_{n+1}^{(1)}\cdot (\alpha_n^{(i-1)} - a_n^{(i-1)}) = \cfrac{\alpha_n^{(i-1)} - a_n^{(i-1)}}{\alpha_n^{(m)} - a_n^{(m)}}, \quad i = 2, ..., m
\end{cases}
\end{equation}
for $n = 0, 1, 2, \ldots$
\end{itemize}
In the following, we will always consider MCFs whose partial quotients are obtained by this algorithm.
In the next proposition, we prove the convergence of the above algorithm. Let us note that it can produce a finite or infinite MCF, like in the real case. In the finite case, the MCF is surely always convergent to a certain $m$--tuple in $\mathbb Q_p$, but the latter can be different from the input $m$--tuple how we will see later. In the infinite case, we need to prove the convergence and we also see that the MCF converges to the input $m$--tuple. 

\begin{proposition} The $MCF$ obtained by the $p$--adic Jacobi--Perron algorithm is convergent. Moreover, if the latter does not stop, then it converges to the input $m$--tuple $(\alpha^{(1)},\ldots, \alpha^{(m)})$.
\end{proposition}
\begin{proof}
The thesis follows from Theorem \ref{thm:conv-alg-s}, considering that in this case we have $a_n^{(m+1)} = 1$ for any $n \geq 0$ and the  Browkin function $s$ satisfies the hypotheses of Theorem \ref{thm:conv-alg-s} by Propostion \ref{prop:sproperties}.
\end{proof}

The next proposition shows the uniqueness of the expansion of a $m$--tuple in MCF satisfying \eqref{eq:cond-conv-v}.

\begin{proposition}
Consider a MCF convergent to the $m$-tuple $(\alpha^{(1)},\ldots, \alpha^{(m)})\in\QQ_p^m$. Suppose that the partial quotients satisfy conditions \eqref{eq:cond-conv-v} and belong to $\mathbb Z[\frac{1}{p}] \cap \left(-\frac{p}{2}, \frac{p}{2} \right)$. Then they are obtained by applying the algorithm \eqref{eq:alg}.
\end{proposition}
\begin{proof} We have to prove that $s(\alpha_n^{(i)})=a_n^{(i)}$ for $n\geq 0$ and $i=1,\ldots, m$. Since $s(a_n^{(i)})=a_n^{(i)}$, by Proposition \ref{prop:sproperties} $a)$ it suffices to show that $|\alpha_n^{(i)}-a_n^{(i)}|<1$, that is $\left |\frac {\alpha_{n+1}^{(i+1)}}{\alpha_{n+1}^{(1)}}\right |<1$ (where we put $\alpha_n^{(m+1)}=1$). This is true observing that $|\alpha_n^{(1)}| = |a_n^{(1)}| \geq 1$ and $|\alpha_n^{(i)}| < |\alpha_n^{(1)}|$, for $i=2, \ldots, m$.
\end{proof}

In the next proposition, we start to study the finiteness of the $p$--adic Jacobi--Perron algorithm proving that when it stops, then the input $m$-tuple consists of elements that are $\QQ$-linearly dependent. In the next section, we will prove that starting with a $m$-tuple of rational numbers, then the algorithm always stops in a finite number of steps.

\begin{proposition} \label{prop:lin-dip}
If the $p$--adic Jacobi--Perron algorithm stops in a finite number of steps when processes the $m$--tuple $\alpha^{(1)},\ldots , \alpha^{(m)} \in \mathbb Q_p$, then $1,\alpha^{(1)},\ldots , \alpha^{(m)}$ are $\QQ$-linearly dependent. 
\end{proposition}
\begin{proof} Assume that the iteration of equations \eqref{eq:alg} stops at the step $n$. We prove the thesis by induction on $n$.
If $n=0$ then $\alpha^{(m)}=a_0^{(m)}\in\QQ$. Now, suppose that the claim is true for every stopping at steps $m<n$. Notice that the $p$--adic Jacobi--Perron algorithm stops at the step $n-1$ when applied to the $m$-tuple $\alpha^{(1)}_1,\ldots,\alpha^{(m)}_1$. Thus, by inductive hypothesis there is a relation
$$c_1\alpha^{(1)}_1+\ldots +c_m\alpha^{(m)}_1+c_{m+1}=0$$  
for some $c_1,\ldots, c_{m+1}\in\QQ$.
By replacing $\alpha^{(1)}_1=\frac 1 {\alpha^{(m)}-a_0^{(m)}}$ and $\alpha^{(i+1)}_1=\frac {\alpha^{(i)}-a_0^{(i)}} {\alpha^{(m)}-a_0^{(m)}}$
we get
$$c_1+c_2(\alpha^{(1)}-a_0^{(1)})+\ldots + c_m(\alpha^{(m-1)}-a_0^{(m-1)})+c_{m+1}(\alpha^{(m)}-a_0^{(m)})=0.$$
\end{proof}

We can observe that there are situations where the $p$--adic Jacobi--Perron algorithm provides a finite MCF that does not converge to the input $m$--tuple, but it converges to some other $m$--tuple of rationals. We provide a simple example of this situation.

\begin{example}
In the case $m=2$ consider
$$\alpha^{(1)}_0= 1+\frac p {p^2+1}, \quad \alpha^{(2)}_0= 1-\frac p {p^2+1}.$$
Then the equations \eqref{eq:alg} produce recursively
\begin{eqnarray*}  &a^{(1)}_0= 1,\quad\quad & a^{(2)}_0 =1\\
& \alpha^{(1)}_1=-\frac{p^2+1}{p},& \alpha^{(2)}_1=-1\\
 & a^{(1)}_1=- \frac 1 p,\quad & a^{(2)}_1 =-1\\
\end{eqnarray*}
Since $a^{(2)}_1=\alpha^{(2)}_1$, the $p$--adic Jacobi--Perron algorithm stops for the input $(\alpha^{(1)}_0,\alpha^{(2)}_0)$, returning the two following sequences of partial quotients:
$$ (1,-\frac 1 p),\quad (1,-1).$$
However, this finite MCF converges to the values
$$1+\frac {-1}{-\frac 1 p}=1+p,\quad 1+\frac {1}{-\frac 1 p}=1-p.$$
\end{example}

\subsection{The generalized p--adic Euclidean algorithm}

It is well--known that the classical algorithm for determining the continued fraction expansion of a real number can be derived by the Euclidean algorithm. Indeed, if we consider $x_0 \geq x_1 >0$ we can write $x_0 = a_0 x_1 + x_2$, where $a_0 = [x_0 / x_1]$ and $0 \leq x_2 < x_1$. Iterating the procedure we get the following sequence of equations
$$x_i = a_i x_{i+1} + x_{i+2}, \quad i = 0, 1, 2, ...$$
where $a_i =[x_i / x_{i+1}]$. If we set $\alpha_i = \frac{x_i}{x_{i+1}}$, we obtain the continued fraction of $\alpha_0$, since
$$\cfrac{x_i}{x_{i+1}} = a_i + \cfrac{x_{i+2}}{x_{i+1}}$$
and consequently we get
$$a_i = [\frac{x_i}{x_{i+1}}] = [\alpha_i], \quad \alpha_{i+1} = \cfrac{1}{\alpha_i - a_i}.$$
Similarly, the Jacobi--Perron algorithm can be obtained from the generalized Euclidean algorithm (see, e.g., \cite{Bry}).
In this section, we see that also the $p$--adic MCFs can be introduced starting from a generalized Euclidean algorithm defined in $\mathbb Q_p$. For this purpose, we need to define a $p$--adic Euclidean algorithm. In \cite{Lag}, the author introduced a $p$--adic division algorithm that we recall in the next theorem.

\begin{theorem}\label{teo:p-adicDivisionAlgorithm} Let $p$ be an odd prime, given given any $\sigma$ and $\tau$ in $\QQ_p$, where $\tau\not=0$, there exists uniquely $q\in\ZZ[\frac 1 p]$ with $|q|_\infty<\frac p 2$ and $\eta\in\QQ_p$ with $|\eta|<|\tau|$ such that 
$$\sigma=q\tau+\eta.$$
\end{theorem}

\begin{remark}
This construction has been also used in \cite{Err} considering different restrictions on the quotient and remainder. 
\end{remark}

Let us note that the quotient $q$ in the previous theorem coincide with $s(\sigma / \tau)$. Indeed, from $\frac{\sigma}{\tau} = q + \frac{\eta}{\tau}$ with $\lvert \frac{\eta}{\tau} \rvert < 1$, we obtain by Proposition \ref{prop:sproperties}, property (b), that
$$s\left( \frac{\sigma}{\tau} \right) = s\left( q + \frac{\eta}{\tau} \right) = s(q) = q.$$
In other words, the function $s$ has the same role as the floor function in the classical division algorithm.

Starting by $\eta_0, \eta_1 \in \QQ_p$, with $\eta_1\not=0$, one can iteratively apply Theorem \ref{teo:p-adicDivisionAlgorithm} in order to get the equations:
\begin{equation*}
\eta_i = q_i \eta_{i+1} + \eta_{i+2}, \quad i = 0, 1, 2, ...
\end{equation*}
with $q_i \in \ZZ[\frac 1 p]$ such that $|q_i|_\infty<\frac p 2$ and $\eta_i, \in\QQ_p$ with $|\eta_i|<|\eta_{i-1}|$. In other words, we get a $p$--adic Euclidean algorithm that either continues indefinitely or stops when $\eta_i=0$ for some $i$. 

Given a $(m+1)$--tuple $x_0^{(1)}, \ldots, x_0^{(m+1)} \in \mathbb Q_p$, we can use the $p$--adic Euclidean algorithm to obtain

\begin{equation}\label{eq:xdue}
x_n^{(i-1)}=a_n^{(i-1)}x_n^{(m+1)}+x_{n+1}^{(i)},
\end{equation}
and $x_{n+1}^{(1)} = x_n^{(m+1)}$, with $i = 2, \ldots, m$, $n = 0, 1, \ldots$, where $a_n^{(i)}=s\left(\frac {x_n^{(i)}}{x_n^{(m+1)}}\right).$ This equations yields to a generalized $p$--adic Euclidean algorithm.
Hence, if we put 
$$\alpha_n^{(i)}=\frac {x_n^{(i)}}{x_n^{(m+1)}},\quad i=1,\ldots, m$$
the $p$--adic Jacobi--Perron algorithm \eqref{eq:alg} 
can be described by the following rules
\begin{equation}\label{eq:Jacobimult}
\left\{ \begin{array}{lll}
 x^{(1)}_{n+1}&=& x^{(m+1)}_n\\
 x_{n+1}^{(i)}&=& x_n^{(i-1)}-a_n^{(i-1)}x_n^{(m+1)}\quad\hbox{for } i=2,\ldots, m+1
 \end{array}
\right .
\end{equation}
In other words, applying the generalized $p$--adic Euclidean algorithm to a $(m+1)$--tuple $x_0^{(1)}, \ldots, x_0^{m+1} \in \mathbb Q_p$ is equivalent to apply the $p$--adic Jacobi--Perron algorithm to the $m$--tuple $\frac{x_0^{(1)}}{x_0^{m+1}}, \ldots, \frac{x_m^{(1)}}{x_0^{m+1}}$. On the other hand, applying the $p$--adic Jacobi--Perron algorithm to a $m$--tuple $\alpha_0^{(1)}, \ldots, \alpha_0^{(m)} \in \mathbb Q_p$ is equivalent to apply the generalized $p$--adic Euclidean algorithm to the $(m+1)$--tuple $\alpha_0^{(1)}, \ldots, \alpha_0^{(m)}, 1$.
Exploiting these considerations, we can prove that the $p$--adic Jacobi--Perron algorithm terminates in a finite number of steps when a $m$--tuple of rational numbers is given as input.

\begin{theorem} Given $\alpha_0^{(1)},\ldots , \alpha_0^{(m)}\in\QQ$ as inputs for the $p$--adic Jacobi--Perron algorithm, then it terminates in a finite number of steps.
\end{theorem}
\begin{proof} 
Since $\alpha_0^{(1)},\ldots , \alpha_0^{(m)}\in\QQ$, we can write them as $\alpha_0^{(i)}= \frac{n_0^{(i)} }{d_0^{(i)}}$, with $n_0^{(i)}, d_0^{(i)}\in\ZZ$ coprimes, for any $i = 1, ..., m$. Considering $\ell = lcm\{d_0^{(1)},\ldots , d_0^{(m)}\}$, we can write $\alpha_0^{(i)} = \cfrac{n_0^{(i)}\cdot \frac{\ell}{d_0^{(i)}}}{\ell}$ and we define the integer numbers
$$x_0^{(m+1)} = \ell, \quad x_0^{(i)} = n_0^{(i)} \cdot \frac{\ell}{d_0^{(i)}}, \quad i = 1, \ldots, m.$$
Thus, applying the $p$--adic Jacobi Perron algorithm to $\alpha_0^{(1)},\ldots , \alpha_0^{(m)}$ is equivalent to apply the generalized $p$--adic Euclidean algorithm to $x_0^{(1)}, \ldots, x_0^{(m+1)}$, which will produce the sequences $(x_n^{(1)})_{n \geq 1}, \ldots, (x_n^{(m+1)})_{n \geq 1}$ whose elements are all in $\mathbb Z\left[ \frac{1}{p} \right]$.  Now, we prove that in this case the generalized $p$--adic Euclidean algorithm terminates in a finite number of steps.
From equations (\ref{eq:Jacobimult}) we get
\begin{eqnarray*}
x_{n+1}^{(m+1)} &=& x_n^{(m)} -a_n^{(m)}x_n^{(m+1)}\\
&=& x_{n-1}^{(m-1)} -a_{n-1}^{(m-1)}x_{n-1}^{(m+1)} -a_n^{(m)}x_n^{(m+1)}\\
&=& x_{n-2}^{(m-2)} -a_{n-2}^{(m-2)}x_{n-2}^{(m+1)}  -a_{n-1}^{(m-1)}x_{n-1}^{(m+1)} -a_n^{(m)}x_n^{(m+1)}\\
&\vdots & \\
&=& x^{(1)}_{n-(m-1)}- a^{(1)}_{n-(m-1)}x^{(m+1)}_{n-(m-1)}-a^{(2)}_{n-(m-2)}x^{(m+1)}_{n-(m-2)}-\ldots -a_n^{(m)}x_n^{(m+1)}\\
&=& x^{(m+1)}_{n-m}- a^{(1)}_{n-(m-1)}x^{(m+1)}_{n-(m-1)}-a^{(2)}_{n-(m-2)}x^{(m+1)}_{n-(m-2)}-\ldots -a_n^{(m)}x_n^{(m+1)}
\end{eqnarray*}
We can write $x^{(m+1)}_n= p^{e_n}h_n$ with $e_n, h_n\in\ZZ$, $p\not | h_n$ and by equation (\ref{eq:xdue}) we have $|x^{(m+1)}_{n+1}|<|x^{(m+1)}_{n}|$ for every $n$, so that $e_{n+1}>e_n$ for every $n$. In particular $e_{n+k} -e_n\geq k$, for every $n,k$. 
Dividing by $p^{e_{n+1}}$ the above equality we obtain
\begin{equation*}\label{eq:eq}
h_{n+1}=\frac{h_{n-m}}{p^{e_{n+1}-e_{n-m}}}- a^{(1)}_{n-(m-1)}\frac{h_{n-(m-1)}}{p^{e_{n+1}-e_{n-(m-1)}}}- a^{(2)}_{n-(m-2)}\frac{h_{n-(m-2)}}{p^{e_{n+1}-e_{n-(m-2)}}}-\ldots - a^{(m)}_{n}\frac{h_{n}}{p^{e_{n+1}-e_{n}}}
\end{equation*}
Moreover, recalling that $|a_n^{(i)}|_\infty < \frac p 2$ for every $i,n$, we have
\begin{eqnarray*}
|h_{n+1}|_\infty &\leq & \frac{|h_{n-m}|_\infty }{p^{e_{n+1}-e_{n-m}}}+|a^{(1)}_{n-(m-1)}|_\infty \frac{|h_{n-(m-1)}|_\infty }{p^{e_{n+1}-e_{n-(m-1)}}}+ |a^{(2)}_{n-(m-2)}|_\infty \frac{|h_{n-(m-2)}|_\infty} {p^{e_{n+1}-e_{n-(m-2)}}}+\ldots + |a^{(m)}_{n}|_\infty \frac{|h_{n}|_\infty}{p^{e_{n+1}-e_{n}}}\\
&\leq & \frac{|h_{n-m}|_\infty }{p^{m+1}}+ |a^{(1)}_{n-(m-1)}|_\infty \frac{|h_{n-(m-1)}|_\infty }{p^{m}}+|a^{(2)}_{n-(m-2)}|_\infty \frac{|h_{n-(m-2)}|_\infty} {p^{m-1}}+\ldots + |a^{(m)}_{n}|_\infty \frac{|h_{n}|_\infty}{p}\\
& < & \frac{|h_{n-m}|_\infty }{p^{m+1}}+  \frac{|h_{n-(m-1)}|_\infty }{2p^{m-1}}+ |\frac{|h_{n-(m-2)}|_\infty} {2p^{m-2}}+\ldots +  \frac{|h_{n}|_\infty}{2}\\
&\leq & \max_{n-m\leq j\leq n}\{h_j\}\left ( \frac 1{p^{m+1}}+\frac 1 2\sum_{k=0}^{m-1} \frac 1 {p^k}\right )\\
&<& \max_{n-m\leq j\leq n}\{h_j\}.\end{eqnarray*}
The last inequality follows from the fact that 
\begin{equation*}
\frac 1{p^{m+1}}+\frac 1 2\sum_{k=0}^{m-1} \frac 1 {p^k}= \frac 1{p^{m+1}}+\frac 1 2 \cdot \frac {1-\frac 1 {p^m}}{1-\frac 1 p} < 1
\end{equation*}
Thus, for $n$ sufficiently large, it must be $h_n=0$, i.e., the algorithm terminates in a finite number of steps.
\end{proof}

\subsection{Examples}

We conclude providing some examples concerning the application of the $p$--adic Jacobi--Perron algorithm.

\noindent Taking $m=2$, in $\mathbb Q_5$, equations \eqref{eq:alg}, applied to the couple of rational numbers $\left( \frac{23}{5}, \frac{14}{19} \right)$ and $\left( \frac{7}{3}, \frac{11}{20} \right)$, provide the following expansions in MCFs:

$$\left( \frac{23}{5}, \frac{14}{19} \right) = \left[ \left( -\frac{2}{5}, \frac{6}{5}, \frac{6}{5}, \frac{4}{5} \right), \left( 1, 1, -1, -1 \right) \right]$$

$$\left( \frac{7}{3}, \frac{11}{20} \right) = \left[ \left( -1, -\frac{4}{5}, -\frac{3}{5} \right), \left( \frac{9}{5}, -1, 0 \right) \right].$$

\noindent If we consider the couple $\left(\sqrt[3]{2}, \frac{5}{4}\right)$, the $p$--adic Jacobi Perron algorithm provides the following finite MCF:

$$\left[\left(1, \frac{4}{5}, \frac{12}{5}\right),\left(0,1,0\right)\right]$$

\noindent that does not converge to the inputs, but it converges to the couple of rationals $\left(\frac{133}{48}, \frac{5}{4}\right)$.

\noindent If we consider two cubic irrationals $\alpha$ and $\beta$, where $\alpha$ is the root largest in modulo of the polynomial $x^3 - \frac{8}{5} x^2 - x - 1$ and $\beta = 1 + \frac{1}{\alpha}$, then the $p$--adic Jacobi--Perron algorithm provides a periodic expansion in MCF:

$$(\alpha, \beta) = \left[ \left( \overline{\frac{8}{5}} \right), \left( \overline{1} \right) \right].$$

\noindent In $\mathbb Q_7$, we can get the following expansions:

$$\left( \frac{31}{16}, \frac{123}{7} \right) = \left[ \left(-2, -\frac{16}{7}, \frac{13}{7}, \frac{17}{7}, \frac{2}{7} \right), \left( -\frac{24}{7}, -2, 2, -2, -1 \right) \right]$$

$$(\gamma, \delta) = \left[ \left( \overline{-\frac{3}{7}} \right), \left( \overline{-2} \right) \right],$$

\noindent where $\gamma$ is the root largest in modulo of $x^3 + \frac{3}{7} x^2 + 2x - 1$ and $\delta = -2 + \frac{1}{\gamma}$.

\noindent Taking $m = 3$, in $\mathbb Q_{11}$, we can get the following expansions:

$$\left( -\frac{5}{4}, \frac{29}{11}, \frac{3}{4} \right) = \left[ \left(-4, \frac{4}{11} \right), \left( \frac{29}{11}, 1 \right), \left( -2, 0  \right) \right]$$

$$\left(-\frac{7}{4}, \frac{2}{5}, \frac{1}{3} \right) = \left[ \left(1, -\frac{3}{11}, -\frac{5}{11}, \frac{4}{11} \right), \left( -4, -2, 0, 0 \right), \left( 4, 1, -4, 0 \right) \right].$$

\end{document}